\documentclass[11pt]{amsart}

\usepackage{amsmath, amscd, amssymb}
\usepackage[frame,cmtip,arrow,matrix,line,graph,curve]{xy}
\usepackage{graphpap, color}
\usepackage[mathscr]{eucal}
\usepackage{cancel}
\usepackage{verbatim}
\usepackage{adjustbox}

\usepackage{tikz}
\usepackage{float}
\usepackage{booktabs}
\usepackage{multirow}
\usepackage{tabularx}

\usepackage{hyperref}

\numberwithin{equation}{section}

\newtheorem{theorem}{Theorem}[section]
\newtheorem{corollary}[theorem]{Corollary}
\newtheorem{proposition}[theorem]{Proposition}
\newtheorem{lemma}[theorem]{Lemma}

\theoremstyle{definition}
\newtheorem{definition}[theorem]{Definition}

\newtheorem{example}[theorem]{Example}
\newtheorem{remark}[theorem]{Remark}

\newcommand{\Q}{\mathbb{Q}}

\newcommand{\C}{\mathbb{C}}

\newcommand{\PP}{\mathbb{P}}

\def\CC{\mathbb{C}}

\def\QQ{\mathbb{Q}}

\def\RR{\mathbb{R}}

\def\ZZ{\mathbb{Z}}

\def\sO{{\mathscr O}}

\newcommand{\cal}{\mathcal}

\def\cF{{\cal F}}

\def\cH{{\cal H}}

\def\cK{{\cal K}}

\def\cP{{\cal P}}

\def\cV{{\cal V}}

\def\tY{{\widetilde{Y}}}

\def\hbar{\overline{h}}

\DeclareMathOperator{\Hom}{Hom}

\def\Ind{\mathrm{Ind} }

\def\dim{\mathrm{dim} }

\def\ch{\mathrm{ch} }

\def\and{\quad{\rm and}\quad}
\def\lra{\longrightarrow }

\def\beq{\begin{equation}}
\def\eeq{\end{equation}}
\def\ben{\begin{enumerate}}
\def\een{\end{enumerate}}

\def\and{\quad\text{and}\quad}

\def\a{\alpha}
\def\b{\beta}

\def\e{\epsilon}

\def\w{\omega}

\def\Fl{\mathrm{Fl}}

\def\Fl{\mathrm{Fl}}
\def\E{\mathrm{E}}
\def\V{\mathrm{V}}
\def\G{\Gamma}

\def\csf{\mathsf{csf}}
\def\asc{\mathsf{asc}}
\def\LLT{\mathrm{LLT}}

\def\diag{\mathrm{diag}}
\def\pr{\mathrm{pr}}
\def\res{\mathsf{res}}

\newcommand{\N}{\mathbb{N}}

\def\bh{\mathbf{h}}

\title[Geometry of twins of Hessenberg varieties]{Geometry of the twin manifolds of regular semisimple Hessenberg varieties and unicellular LLT polynomials}
\date{November 29, 2023}

\author{Young-Hoon Kiem}
\address{School of Mathematics, Korea Institute for Advanced Study, 85 Hoegiro, Dongdaemun-gu, Seoul 02455, Korea}
\email{kiem@kias.re.kr}

\author{Donggun Lee}
\address{Center for Complex Geometry, Institute for Basic Science, 55 Expo-ro, Yuseong-gu, Daejeon 34126, Korea}
\email{dglee@ibs.re.kr}

\thanks{YHK was partially supported by Korea NRF grant 2021R1F1A1046556. DL was supported by the Institute for Basic Science (IBS-R032-D1).}

\begin{document}

\begin{abstract}
Recently, Masuda-Sato and Precup-Sommers independently proved an LLT version of the Shareshian-Wachs conjecture which says that the Frobenius characteristics of the cohomology of the twin manifolds of regular semisimple Hessenberg varieties are unicellular LLT polynomials. 
The purpose of this paper is to study the geometry of twin manifolds and we
prove that they are related by explicit blowups and fiber bundle maps.
Upon taking their cohomology, we obtain a direct proof of the modular law which establishes the LLT Shareshian-Wachs conjecture.
\end{abstract}

\maketitle

\section{Introduction}
LLT polynomials are symmetric functions that serve as $q$-deformations of the product of Schur functions, introduced by Lascoux, Leclerc, and Thibon \cite{LLT} in their study of quantum affine algebras.
A specific class of these polynomials known as unicellular LLT polynomials (Definition \ref{n2}) was explored in \cite{CM} using Dyck paths, or Hessenberg functions, in parallel with the chromatic quasisymmetric functions. 
The purpose of this paper is to investigate on the geometry of the twin manifolds of regular semisimple Hessenberg varieties (Propositions \ref{38} and \ref{39}) and provide a direct geometric proof of the fact that  
unicellular LLT polynomials are the Frobenius characteristics of representations of symmetric groups $S_n$ on the cohomology of the twin manifolds (Theorem \ref{48}).

\medskip

Hessenberg varieties are subvarieties of flag varieties with interesting properties in geometric, representation theoretic and combinatorial aspects (cf.~\cite{DMPS, AH}). 
One of their notable features is the $S_n$-action on their cohomology \cite{Tym}, where the induced graded $S_n$-representations are equivalent to the purely combinatorially defined symmetric functions known as the chromatic quasisymmetric functions \cite{Sta, SW} of the associated indifference graphs. 
This equivalence (cf.~\eqref{57}), known as the Shareshian-Wachs conjecture \cite{SW}, proved in \cite{BC, GP2}, translates the longstanding conjecture by Stanley and Stembridge \cite{SS} on $e$-positivity of the chromatic (quasi)symmetric functions into a positivity problem on the $S_n$-representations on the cohomology of Hessenberg varieties.

A natural question arises whether there exist geometric objects that encode unicellular LLT polynomials through their cohomology, as in the Shareshian-Wachs conjecture. 
Recently, an answer was found by Masuda-Sato and Precup-Sommers in \cite{MS,PS} 
where they proved
that unicellular LLT polynomials are the Frobenius characteristics of the cohomology of the \emph{twin manifolds} of regular semisimple Hessenberg varieties. 

\medskip

The unitary group $U(n)$ is acted on by its maximal torus $T=U(1)^n$ by left and right multiplications. So we have the quotient maps 
$$Y:= T\backslash U(n) \xleftarrow{p_1}U(n)\xrightarrow{p_2}U(n)/T\cong \Fl(n)=:X$$
where $X$ denotes the flag variety $\Fl(n)$ and $Y$ denotes the isospectral manifold of Hermitian matrices with a fixed spectrum (cf.~\S\ref{n46}).  
The twin manifold of a Hessenberg variety $X_h\subset \Fl(n)=X$ is now defined in \cite{AB} as the submanifold $$Y_h:=p_1(p_2^{-1}(X_h))$$ of the isospectral manifold $Y$. 
These twin manifolds $Y_h$, which are the spaces of staircase Hermitian matrices with a fixed given spectrum, are interesting compact orientable smooth real algebraic varieties.
They generalize the space of tridiagonal matrices of a given spectrum \cite{Tom,BFR,DJ} and 
we have natural isomorphisms 
\beq\label{n47}
H^*_T(Y_h)\cong H^*_{T\times T}( p_1^{-1}(X_h))\cong H^*_T(X_h)\eeq 
which induce an $S_n$-action on the cohomology $H^*(Y_h)$ from that on $H^*(X_h)$ in \cite{Tym}.
The LLT analogue of the Shareshian-Wachs conjecture (\emph{LLT-SW conjecture}, for short) tells us that the Frobenius characteristics of $H^*(Y_h)$ are the unicellular LLT polynomials (cf.~Theorem \ref{48}). 
The known proofs in \cite{MS,PS} are rather indirect and use only  the Hessenberg varieties without looking into the geometry of twin manifolds themselves.
See Remark~\ref{56} for more details. 
Therefore, it seems natural to ask for a direct approach through the geometry of twin manifolds.  
  
\medskip

The \emph{modular law} (cf.~ Definition \ref{n42}), introduced in \cite{AN, GP} for chromatic quasisymmetric functions and in \cite{Ale, Lee} for unicellular LLT polynomials, is a significant relation involving specific triples of these functions. It serves as a symmetric function analogue to the well known deletion-contraction relation of chromatic polynomials.
In fact, Abreu and Nigro proved in \cite{AN} that together with an initial condition (for the case of $h(i)=n$ for all $i$) and the multiplicativity (cf.~\eqref{n3}), 
the modular law completely determines the chromatic quasisymmetric functions and unicellular LLT polynomials. 

In \cite{KL}, the authors investigated on the geometry of Hessenberg varieties $X_h$ and proved that 
the Hessenberg varieties $X_{h_-}$, $X_{h}$ and $X_{h_+}$ for a modular triple $\bh=(h_-,h,h_+)$ (cf.~Definition \ref{59}) are related by explicit blowups and projective bundle maps. 
By applying the blowup formula and projective bundle formula, we then immediately obtain
the modular law for the cohomology of $X_h$, which provides us with an elementary proof of the Shareshian-Wachs conjecture. 

\medskip

In this paper, we investigate on the geometry of the twin manifolds $Y_h$. 
The key for our comparison of twin manifolds is the \emph{roof manifold} $\tY_\bh$ defined in Definition \ref{43}. For a modular triple $\bh=(h_-,h,h_+)$ (cf.~Definition \ref{59}), we construct maps
$$\xymatrix{&\tY_\bh \ar[ld]_-{\pi}\ar[rd]^-{\pr_2}& \\ Y_{h_+}&&\PP^1}$$
where $\pr_2$ is a smooth fibration over the complex projective line $\PP^1$ with fiber $Y_h$ (Proposition \ref{38}) and $\pi$ is the blowup along the submanifold $Y_{h_-}$ of complex codimension 2 (Proposition \ref{39}). 
We define an $S_n$-action on the cohomology $H^*(\tY_\bh)$ and show that 
the induced maps on cohomology by $\pi$ and $\pr_2$ are $S_n$-equivariant. 
We thus obtain  $S_n$-equivariant isomorphisms
\[H^*(Y_h)\oplus H^{*-2}(Y_h)\cong H^*(\tY_\bh)\cong 
H^*(Y_{h_+})\oplus H^{*-2}(Y_{h_-}).\]
Upon taking the Frobenius characteristic, we have the modular law for $H^*(Y_h)$ and hence the LLT-SW conjecture
$$\sum_{k\ge 0}\ch (H^{2k}(Y_h)) q^k = \LLT_h(q)$$ 
where $\LLT_h(q)$ denotes the unicellular LLT polynomial associated to $h$.

\medskip

The layout of this paper is as follows. 
In \S\ref{4}, we review the definition of unicelullar LLT polynomials and their characterization by the modular law.
In \S\ref{14}, we review the results in \cite{AB} on twin manifolds including the $S_n$-action defined on their cohomology. 
In \S\ref{3}, we study the geometry of twin manifolds of triples and in \S\ref{S4}, we establish the modular law for $S_n$-representations on their cohomology.

\medskip

All cohomology groups in this paper have rational coefficients. By $\PP^r$, we denote the complex projective space of one dimensional subspaces in $\CC^{r+1}$.

\medskip

\textbf{Acknowledgement.} We thank Anton Ayzenberg, Jaehyun Hong, Antonio Nigro and Takashi Sato for enlightening  discussions and comments.
 
\bigskip

\section{Unicellular LLT polynomials}\label{4}

LLT polynomials are symmetric functions introduced by Lascoux, Leclerc, and Thibon \cite{LLT} as $q$-deformations of the product of Schur functions in their study of quantum affine algebras.  In the case of unicellular LLT polynomials, which form a subfamily of these symmetric functions, a more convenient model is presented in \cite{CM} using Hessenberg functions. This model represents unicellular LLT polynomials as symmetric functions that encode colorings of graphs, which may not be proper.

In this section, we recall the definition of unicellular LLT polynomials in terms of Hessenberg functions from \cite{CM} and the characterization by the modular law from \cite{AN2}. 
\subsection{Definitions}\label{n1}
Unicellular LLT polynomials can be defined as follows. 
\begin{definition}
	Let $[n]:=\{1,\cdots, n\}$ for an integer $n\geq1$.
	\begin{enumerate}
		\item A \emph{Hessenberg function} is a nondecreasing function $h:[n]\to[n]$ satisfying $h(i)\geq i $ for all $i$
		\item The  \emph{indifference graph} $\G_h$ associated to $h$ is the graph whose  set of vertices is $V(\G_h)=[n]$ and whose set of edges is
			\[\E(\G_h)=\{(i,j)\in [n]\times [n]:i<j\leq h(i)\}.\]
	\end{enumerate}
\end{definition}

Every unicellular LLT polynomial can be written as a symmetric function which encodes vertex-colorings of the indifference graph $\G_h$, similar to the definition of the chromatic quasisymmetric function \cite{Sta,SW}.

A map $\gamma:\V(\G_h)\to \N$ is said to be a (vertex-)coloring of $\G_h$, and it is said to be \emph{proper} if and only if $\gamma(i)\neq \gamma(j)$ whenever $(i,j)\in \E(\G_h)$, where $\N$ is the set of colors indexed by  positive integers.

\begin{definition}\label{n2} \cite{CM, Sta, SW}
	Let $h:[n]\to [n]$ be a Hessenberg function. Let $\Lambda$ be the ring of symmetric functions in variables  $x_1,x_2,\cdots$.

	\begin{enumerate}
	\item The \emph{unicellular LLT polynomial} associated to $h$ is 
	\beq\label{n12}
	\LLT_h(q):=\sum_{\substack{\gamma:\V(\G_h)\to \N }}q^{\asc_h(\gamma)}x_{\gamma(1)}\cdots x_{\gamma(n)} \ \ \in \Lambda[q]\eeq
	where $\gamma$ runs over all colorings which are not necessarily proper and 
	\[\asc_h(\gamma):=\lvert\{(i,j)\in \E(\G_h):\gamma(i)<\gamma(j)\}\rvert.\]
	\item  The \emph{chromatic quasisymmetric function} associated to $h$ is 
	\beq\label{n11} 
	\csf_h(q):=\sum_{\substack{\gamma:\V(\G_h)\to \N \\ \text{proper}}}q^{\asc_h(\gamma)}x_{\gamma(1)}\cdots x_{\gamma(n)}\ \ \in \Lambda[q]\eeq
	where $\gamma$ runs over all \emph{proper} colorings.
	\end{enumerate}
\end{definition}

\subsection{Modular law and characterization of $\LLT_{(-)}$}

The symmetric functions $\LLT_h$ and $\csf_h$ satisfy a linear relation called the \emph{modular law}.
\begin{definition}[Modular triple] \label{59} Let $h_-,h,h_+:[n]\to[n]$ be Hessenberg functions.
	The triple $(h_-,h,h_+)$ is called a \emph{modular triple} if it satisfies one of the following. 
		\begin{enumerate}
			\item If $h(j)=h(j+1)$ and $h^{-1}(j)=\{j_0\}$ for some $1\leq j_0<j<n$, then $h_-$ and $h_+$ are defined by 
			\[h_-(i)=\begin{cases}
				j-1 & \text{for }i=j_0\\
				h(i) &\text{otherwise}
				\end{cases}
			\and h_+(i)=\begin{cases}
				j+1 &\text{for }i=j_0\\
				h(i) & \text{otherwise.}
				\end{cases} 
			\]
			\item If $h(j)+1=h(j+1)\neq j+1$ and $h^{-1}(j)=\emptyset$ for some $1\leq j<n$, then $h_-$ and $h_+$ are defined by 
			\[h_-(i)=\begin{cases}
				h(j) &\text{for }i=j+1\\
				h(i) &\text{otherwise}
				\end{cases}  
			\and h_+(i)=\begin{cases}
				h(j)+1 &\text{for }i= j\\
				h(i) &\text{otherwise.}
				\end{cases}
			\]
		\end{enumerate}
\end{definition}

The two conditions (1) and (2) in Definition \ref{59} are actually dual to each other. See Remark \ref{r1}. 

\begin{definition}[Modular law] \label{n42}
Let $F$ be a function from the set of Hessenberg functions to $\Lambda[q]$. We say that $F$ satisfies the \emph{modular law} if
	\beq \label{5} (1+q)F(h)=F(h_+)+qF(h_-)\eeq
	for every modular triple $(h_-,h,h_+)$. 
\end{definition}

Note that unicellular LLT polynomials and chromatic quasisymmetric functions can be viewed as  functions $\LLT_{(-)}$ and $\csf_{(-)}$ from the set of Hessenberg functions to $\Lambda[q]$.
\begin{proposition} 
\cite{AN, Ale, Lee} 
Unicellular LLT polynomials and chromatic quasisymmetric functions satisfy the modular law. 
\end{proposition}

The modular law, analogous to the deletion-contraction property in chromatic polynomials, plays a crucial role in determining these symmetric functions recursively.

For $m\geq 1$, let 
\[[m]_q:=\frac{1-q^m}{1-q}=1+q+\cdots +q^{m-1}, \quad  [m]_q!:=\prod_{i=1}^m [i]_q!.\] We set $[0]_q!:=1$. Moreover, let $e_m:=\sum_{1\leq i_1<\cdots <i_m}x_{i_1}\cdots x_{i_m}\in \Lambda$ be the $m$-th elementary symmetric function.

\begin{theorem} \cite{AN,AN2} \label{46}
	$\LLT_{(-)}$ (resp. $\csf_{(-)}$) is the unique function $F$ from the set of Hessenberg functions to $\Lambda[q]$ satisfying the following.
	\begin{enumerate}
		\item For $h:[n]\to [n]$ with $h(i)=n$ for all $i$, $\cK_n:=F(h)$ satisfies  
		\beq \label{13} \cK_n=\sum_{i=1}^n (q-1)^{i-1}\frac{[n-1]_q!}{[n-i]_q!}e_i \cK_{n-i}, \qquad \cK_0:=1\eeq
		(resp. $\cK_n=[n]_q!e_n$). 
		\item It is multiplicative: when $h(j)=j$ for $1\leq j<n$,
		\beq\label{n3} F(h)=F(h') F(h'')\eeq
		where $h':[j]\to[j]$ and $h'':[n-j]\to [n-j]$ are Hessenberg functions
		defined by
		\[h'(i)=i ~\text{ for }~i\in[j] \and h''(i)=h(i+j)-j ~\text{ for }~i\in [n-j].\] 
		\item It satisfies the modular law \eqref{5}.
	\end{enumerate}
\end{theorem}

\medskip

\begin{remark}
One fundamental technique in representation theory and combinatorics is to 
construct a geometric object corresponding to an object of interest. 
Hard combinatorics problems are often translated into well known geometry problems and solved subsequently, as demonstrated by the recent spectacular works of June Huh.

The Shareshian-Wachs conjecture  formulated in \cite{SW} and proved in \cite{BC, GP2}  tells us 
that the chromatic quasisymmetric function \eqref{n11} is the $\omega$-dual of the Frobenius characteristic 
$$\cF(h):=\sum_{k\ge 0} \ch (H^{2k}(X_h)) q^k \ \ \in \ \  \Lambda[q]$$ 
of regular semisimple Hessenberg varieties $X_h$ in \S\ref{n29} below. Here $\omega$ is an involution of $\Lambda$ interchanging each Schur function with its transpose.  
The first two conditions (1) and (2) in Theorem \ref{46} for $\csf$ are easy to check for $\omega \cF(h)$ and hence the Shareshian-Wachs conjecture follows immediately from the modular law \eqref{5} for $\cF(h)$. 
In \cite{KL}, we investigated on the geometry of generalized Hessenberg varieties and constructed canonical $S_n$-equivariant isomorphisms
$$H^{2k}(X_h)\oplus H^{2k-2}(X_h)\cong H^{2k}(X_{h_+})\oplus H^{2k-2}(X_{h_-}).$$
Upon taking the Frobenius characteristic, we obtain the modular law \eqref{5} and hence the Shareshian-Wachs conjecture $$\cF(h)=\omega\, \csf_h(q).$$ 

In the remainder of this paper, we will prove that the three conditions in Theorem \ref{46} for unicellular LLT are satisfied for the Frobenius characteristics of representations of $S_n$ on the cohomology of the twin manifolds $Y_h$ of regular semisimple Hessenberg varieties $X_h$, by finding geometric relations among the twin manifolds that give rise to the modular law \eqref{5} upon taking cohomology. This will give us a direct proof of the LLT-SW conjecture (cf.~Theorem \ref{48}) without relying on the Shareshian-Wachs conjecture.
\end{remark}

\bigskip

\section{Twin manifolds and their cohomology}\label{14}
In this section, we collect necessary facts about the twin manifolds $Y_h$ of regular semisimple Hessenberg varieties $X_h$ of type A 
from \cite{AB}.

Let $h:[n]\to [n]$ be a Hessenberg function (Definition \ref{n1}) where $[n]=\{1,\cdots, n\}$. 
Let $$x=\mathrm{diag}(\lambda_1,\cdots,\lambda_n),\quad
\lambda_1 >\lambda_2 >\cdots > \lambda_n \in \RR$$
be a fixed regular semisimple diagonal matrix. 
Let $T\cong U(1)^n$ denote the group of diagonal unitary matrices.

\subsection{Isospectral manifolds} \label{n46} 
Let $\mathcal{H}$ denote the real vector space of $n\times n$ Hermitian matrices. 
Let $Y=Y(x)\subset \mathcal{H}$ be the set of $n\times n$ Hermitian matrices whose characteristic polynomial is $\prod_{i=1}^n(t-\lambda_i)$.  In other words, $Y$ is the set of $n\times n$ Hermitian matrices with fixed (unordered) spectrum $\{\lambda_i\}$. 
As the diffeomorphism type of $Y$ is independent of $x$ by \cite[Theorem 3.5]{AB}, we will suppress $x$ to simplify the notation.   

By the spectral decomposition theorem in linear algebra, any matrix in $Y$ is of the form
$g^{-1}xg,$ with $g\in U(n)$
and the map  \[ U(n)\lra Y\subset \cH,  \quad g\mapsto g^{-1}xg\]
induces a diffeomorphism 
\beq\label{n7} T\backslash U(n)\cong Y,\quad Tg\mapsto g^{-1}xg.\eeq
In particular, $Y$ is a compact smooth orientable manifold of real dimension $n^2-n$. 

Let $X:=\Fl(n)$ denote the variety of flags of $\CC$-linear subspaces
$$(V_1\subset V_2\subset\cdots\subset V_n=\CC^n), \quad \dim\, V_i=i$$
which is a smooth projective variety of real dimension $n^2-n$. 
As the columns of a unitary matrix define a flag in $\CC^n$, we have 
$$X=\Fl(n)\cong U(n)/T.$$

The isospectral manifold $Y$ and the flag variety $X$ fit into the following diagram
\beq\label{n4}
Y\cong T\backslash U(n)\xleftarrow{~p_1~}U(n)\xrightarrow{~p_2~}U(n)/T\cong \Fl(n)=X
\eeq
where $p_1$ is the left quotient and $p_2$ is the right quotient.

\subsection{Hessenberg varieties and their twins}\label{n29}

For a Hessenberg function $h:[n]\to [n]$, the \emph{Hessenberg variety} associated to $h$ is
defined in \cite{DMPS} as 
\[X_h:=\{(V_1\subset V_2\subset \cdots \subset  V_{n}= \C^n)\in X : xV_i\subset V_{h(i)} ~\text{ for all }~ i\}.\]
Under our assumptions, by \cite{DMPS}, 
the Hessenberg variety $X_h$ is a smooth projective variety of real dimension 
\beq\label{n6}
2\, \sum_{i=1}^n (h(i)-i).\eeq
See \cite{AH} for a recent survey on Hessenberg varieties.

The \emph{twin manifold} $Y_h$ of $X_h$ is a submanifold of $Y$ defined in \cite{AB}  by 
\beq\label{n5} Y_h:=p_1(Z_h), \quad Z_h:=p_2^{-1}(X_h)\eeq
where $p_1$ and $p_2$ are the quotient maps in \eqref{n4}.

By \cite[Theorem 3.5]{AB},
$Y_h$ is a compact real smooth manifold of dimension \eqref{n6} whose diffeomorphism type is independent of the choice of $x$.
Using \eqref{n7}, it is straightforward to check that $Y_h$ is precisely, the locus of staircase matrices 
\beq \label{25} Y_h=\{y\in Y~:~y_{ij}=0 \text{ if either }i>h(j) \text{ or }j>h(i)\}\eeq
where $y_{ij}=\overline{y_{ji}}$ denotes the entry of the Hermitian matrix $y$ at the $i$-th row and $j$-th column. 

Since the real dimension of $Y$ is $n^2-n=2 \sum_{i=1}^{n}(n-i)$ and $Y_h$ is defined by the vanishing of $\sum_{i=1}^{n}(n-h(i))$ complex valued functions by \eqref{25}, the expected dimension 
$$2\sum_{i=1}^n(n-i)-2\sum_{i=1}^n(n-h(i))$$
of $Y_h$ coincides with the actual dimension \eqref{n6}. 
In particular, $Y_h$ is the {transversal} vanishing locus of 
\beq \label{23}f_{ij}:
Y\lra \C, \quad y\mapsto y_{ij}\eeq
where $(i,j)$ runs over the pairs with $h(i)<j$, or equivalently $h(j)<i$. 

By \eqref{n7}, we have a right action of $T$ on $Y$ by 
\beq\label{n8} Y\times T\lra Y,\quad y\cdot t=t^{-1}yt.\eeq
If we let $t=\diag(t_1,\cdots, t_n)\in T$, then 
\beq\label{24}  f_{ij}(y\cdot t)=t_i^{-1}t_jf_{ij}(y).\eeq 
In particular, we have an induced $T$-action on $Y_h$ for every Hessenberg function $h$, 
whose fixed point set is exactly 
\beq\label{n10}
Y_h^T=\{\diag(\lambda_{\sigma(1)},\cdots,\lambda_{\sigma(n)})\,|\, \sigma\in S_n\}\cong S_n.\eeq

Note that our notation is different from that in \cite{AB}, where Hessenberg varieties and their twin manifolds are denoted by $Y_h$ and $X_h$ respectively.

\subsection{Goresky-Kottwitz-MacPherson theory} 
When a manifold admits a nice torus action, we can compute its equivariant cohomology 
from the data of 0 and 1-dimensional orbits. 

\begin{definition}\label{n9}
(See \cite[Definition 5.1]{AB}.)
A compact orientable manifold $M$ with a smooth action of a compact torus $T=U(1)^n$ is called a \emph{GKM manifold} if it satisfies the following conditions. 
	\begin{enumerate}
		\item $M$ is equivariantly formal.
		\item The set $M^T$ of $T$-fixed points is finite.
		\item The weights of the induced representation of $T$ on the tangent space of $M$ at each $y\in M^T$ are pairwisely non-collinear: if
		\[\mathbb{T}_M|_y \cong \bigoplus_{i=1}^m \C \a_i, \quad \a_i\in \Hom (T,U(1))\cong \ZZ^n, \]
			then $\a_i$ and $\a_j$ are non-collinear as vectors 
			whenever $i\neq j$.
		\item Every 2-dimensional $T$-invariant closed submanifold which is the union of $T$-orbits of dimension at most one, has a $T$-fixed point. 
	\end{enumerate}
\end{definition}
By (2)--(4) in Definition \ref{n9}, the \emph{1-skeleton} of $M$, which is by definition the union of 0 or 1-dimensional $T$-orbits, is the union of $M^T$ and $T$-invariant 2-spheres. The induced $T$-action on each $T$-invariant 2-sphere $S^2\cong \PP^1$ is of the form
\[T\times \PP^1\lra \PP^1, \qquad (t,[z_0:z_1])\mapsto [z_0:\a(t)z_1]\]
for some $\a\in \Hom(T,U(1))\cong \ZZ^n$. In particular, each $T$-invariant 2-sphere connects precisely two $T$-fixed points, with the associated weight $\a$ determined uniquely up to sign.

The GKM theory \cite{GKM, Kur} tells us that the $T$-equivariant cohomology $H^*_T(M)$ of a GKM manifold $M$ is determined by the combinatorial data of its 1-skeleton as a subring of the $T$-equivariant cohomology $H^*_T(M^T)$ of its 
$T$-fixed point set. 
Indeed, by torus localization, 
$H^*_T(M)$ is embedded into 
\[H^*_T(M^T) 
\cong \bigoplus_{y\in M^T}\QQ[t_1,\cdots, t_n]\] 
by the pullback homomorphism induced by the inclusion $M^T\subset M$.
\begin{theorem} 
\cite[Theorem 5.2]{AB} \label{35} 
The image of $H^*_T(M)$ in $H^*_T(M^T)$ is 
	\[H^*_T(M)\cong \{(f_y)_{y\in M^T}: f_{y_e}\equiv f_{y_e'} \mathrm{~modulo~}\a_e\}\]
	as an algebra over $H^*_T:=H^*_T(\mathrm{pt})\cong \QQ[t_1,\cdots,t_n]$, where $f_y\in H^*_T(\{y\})\cong \QQ[t_1,\cdots,t_n]$, $e$ runs over the set of $T$-invariant 2-spheres in $M$, $\{y_e,y_e'\}$ is the set of $T$-fixed points in $e$ and $\a_e$ is the $T$-weight associated to $e$.
\end{theorem}

\subsection{Equivariant cohomology of $Y_h$} 
For the $T$-action on $Y_h$ by \eqref{n8}, we may use the GKM theory to investigate the cohomology of $Y_h$.

\begin{theorem} \label{34}
	$Y_h$ is a GKM manifold by the following.
	\begin{enumerate}
	\item The cohomology groups of $Y_h$ vanish in odd degrees. In particular, $Y_h$ is equivariantly formal.
	\item The set of $T$-fixed points is $Y_h^T\cong S_n$.
	\item Two $T$-fixed points $\sigma, \tau \in S_n$ are connected by a $T$-invariant 2-sphere in $Y_h$ if and only if $\tau=\sigma\cdot (i,j)$ for some $i<j\leq h(i)$, where  
		$(i,j)$ denotes the transposition interchanging $i$ and $j$.
	\item The tangent space $\mathbb{T}_{Y_h}|_y$ of $Y_h$ at $y\in Y_h^T$ is isomorphic to 
\[\mathbb{T}_{Y_h}|_y \cong \bigoplus_{i<j\leq h(i)}\C \e_{ij}, \quad \e_{ij}=\e_i-\e_j\]
 as a $T$-representation, where $\e_i\in \Hom(T,U(1))$ denotes the character of $T$ sending $(t_1,\cdots,t_n)\in T$ to $t_i$.
\end{enumerate}
\end{theorem}

\begin{proof}
	For (1), see \cite[Theorem 3.5 and Remark 3.7]{AB}. \eqref{n10} gives (2). 
	For (3), see below the proof of \cite[Proposition 3.9]{AB}.
	For (4), see \cite[Proposition 3.9]{AB} and its proof in page 16687. 
\end{proof}
By Theorem \ref{35}, we thus have the following description of $H^*_T(Y_h)$, via the  restriction map
\beq \label{21}\res:H^*_T(Y_h)\hookrightarrow H^*_T(Y_h^T)\cong \bigoplus_{v\in S_n}\QQ[t_1,\cdots, t_n]\eeq
induced by the inclusion $Y_h^T\subset Y_h$.
\begin{proposition}\label{n32} \cite[Proposition 5.3]{AB}
\beq \label{6} H^*_T(Y_h)\cong \left \{(p_v)_{v\in S_n}:  
p_v\equiv p_{v\cdot (i,j)} \mathrm{~modulo~} t_i-t_j \ \ \forall  i<j\leq h(i)\right \}\eeq
 as algebras over $H^*_T=\Q[t_1,\cdots, t_n]$, where $p_v\in \Q[t_1,\cdots, t_n]$.
 \end{proposition}

Recall that the $T$-equivariant cohomology of the Hessenberg variety $X_h$ admits a similar description in \cite{Tym} (see also \cite[\S8]{AHHM})
\beq\label{22} H^*_T(X_h)\cong \left \{(p_v)_{v\in S_n}: 
p_v\equiv p_{v\cdot (i,j)} \mathrm{~modulo~} t_{v(i)}-t_{v(j)} \ \ \forall  i<j\leq h(i)\right \}.\eeq
Comparing \eqref{6} with \eqref{22}, we find that the ring isomorphism
\beq \label{15} \xi: \prod_{v\in S_n}\Q[t_1,\cdots ,t_n]\lra \prod_{v\in S_n}\Q[t_1,\cdots ,t_n], \quad (p_v)_{v\in S_n}\mapsto (vp_v)_{v\in S_n}\eeq
restricts to an isomorphism 
\beq\label{n49}
\xi: H^*_T(Y_h)\lra H^*_T(X_h)\eeq 
of subrings \cite[p.16689]{AB} where $vp_v$ denotes the action of $v$ on $p_v$ by \eqref{n13}. 
Indeed, if $p_v\equiv p_{v\cdot (i,j)}$ modulo $t_i-t_j$, then 
\[vp_v-v\cdot (i,j)p_{v\cdot (i,j)}=v(p_v-p_{v\cdot (i,j)})+v(p_{v\cdot (i,j)}-(i,j)p_{v\cdot (i,j)})\equiv 0\]
modulo $v(t_i-t_j)=t_{v(i)}-t_{v(j)}$. Note that for monomials $f=t_1^{a_1}\cdots t_n^{a_n}$, 
\[f-(i,j)f=(t_i^{a_i}t_j^{a_j}-t_i^{a_j}t_j^{a_i})\prod_{k\neq i,j}t_k^{a_k}\equiv 0~\mathrm{~modulo~}~t_i-t_j.\] 

\begin{remark}
	The $T$-weights on the tangent spaces of the fixed points are well defined only up to sign. We use the sign choices in Theorem~\ref{34} (4). (See \cite[(10) and p.16687]{AB}.)
\end{remark}
\begin{remark}\label{n14}
	One can check that the isomorphism $\xi$ is in fact the natural one given by 
	$$H^*_T(Y_h)\cong H^*_{T\times T}(Z_h)\cong H^*_T(X_h)$$ 
	where $Z_h=p_2^{-1}(X_h)\subset U(n)$ in \eqref{n5} is acted on by $T\times T$ by left and right multiplications so that $Y_h=T\backslash Z_h$ and $X_h=Z_h/T$. 
	If we identify $S_n$ with the group of permutation matrices in $U(n)$ so that $TS_n=S_nT$, 
	one can easily check that \eqref{15} is given by the natural isomorphism 
	$$H^*_T(Y_h^T)\cong H^*_{T\times T}(TS_n)=H^*_{T\times T}(S_nT)\cong H^*_T(X_h^T),$$ 
	where $Y_h^T=T\backslash TS_n$ and $X_h^T=S_nT/T$.
\end{remark}

\subsection{An $S_n$-action on $H^*(Y_h)$} 
In \cite{Tym}, Tymoczko defined an action of $S_n$ on the cohomology of the Hessenberg variety $X_h$ whose Frobenius characteristic turns out to coincide with the chromatic quasisymmetric function \eqref{n11} by \cite{BC, GP2}.
Similarly, there is a natural action of $S_n$ on the cohomology of $Y_h$.

Let us first recall the dot action on $H^*(X_h)$. Consider the $S_n$-action on $H^*_T(X_h^T)\cong \bigoplus_{v\in S_n}\QQ[t_1,\cdots, t_n]$ defined by
\beq \label{19} (\mu,(p_v)_{v\in S_n})\mapsto (\mu p_{\mu^{-1}v})_{v\in S_n}, \quad \text{ for }~\mu\in S_n\eeq
where \beq\label{n13}
(\mu p)(t_1,\cdots,t_n):=p(t_{\mu(1)},\cdots t_{\mu(n)})\eeq 
for $p\in \QQ[t_1,\cdots, t_n]$. From \eqref{22}, it is straightforward to check that this action preserves the subring $H^*_T(X_h) \subset H^*_T(X_h^T)$. 
Hence we have an induced action of $S_n$ on $H^*_T(X_h)$ which in turn   
defines an $S_n$-action on 
\[ 
H^*(X_h)\cong H^*_T(X_h)/\mathfrak{m}H^*_T(X_h),\] 
as it preserves the submodule generated by $\mathfrak{m}:=(t_1,\cdots,t_n)$. This is called \emph{Tymoczko's dot action}.

\begin{definition} \cite{AB}
	By the isomorphism $\xi$ in \eqref{n49}, the dot action on $H^*_T(X_h)$ pulls back to an action of $S_n$ on $H^*_T(Y_h)$ defined by
\beq \label{1}(\mu,(p_v)_{v\in S_n})\mapsto (p_{\mu^{-1}v})_{v\in S_n} \quad \text{ for }~\mu\in S_n.\eeq
As this preserves $\mathfrak{m}H^*_T(Y_h)$, 
\eqref{1} defines an action of $S_n$ on 
\beq\label{n31}
H^*(Y_h)\cong H^*_T(Y_h)/\mathfrak{m}H^*_T(Y_h).\eeq
This is called the \emph{dagger action} in \cite{MS}.
\end{definition}

Using this dagger action, we have the following. 
\begin{definition} \label{30}
For a graded $S_n$-module $V=\bigoplus_{k\ge 0} V_k$, we define
$$\ch_q(V)=\sum_{k\ge 0} \ch(V_k) q^k\in \Lambda[q]$$
where  
	$\ch$ is the Frobenius characteristic from the ring of representations of symmetric groups $S_n$ onto the ring $\Lambda$ of symmetric functions \cite[\S I.7]{Mac}.
For a Hessenberg function $h:[n]\to [n]$, we define
	\[\cP(h):=\sum_{k\geq 0}\ch(H^{2k}(Y_h))q^k=\ch_q(H^{2*}(Y_h)) \in \Lambda[q].\]
\end{definition}

\begin{example}\label{33}
	Suppose $h(i)=n$ for all $1\leq i \leq n$ so that $Y_h$ is the isospectral manifold $Y\cong T\backslash U(n)$ in \eqref{n7} and $X_h$ is the flag variety $$X=\Fl(n)\cong U(n)/T.$$
	We identify $S_n$ with the group of permutation matrices in $U(n)$. 
	In this case, we have a left (resp. right) action of $S_n$ on $X$ (resp. $Y$) by left (resp. right) multiplication of permutation matrices. 
	
	Let $\cV_i$ denote the rank $i$ tautological vector bundle on the flag variety $X$. As the cohomology ring of $X$ is generated by the line bundles $\cV_i/\cV_{i-1}$ and the $S_n$ action preserves the line bundles, we find that the action of $S_n$ on $H^*(X)$ is trivial (cf.~\cite[Proposition 4.2]{Tym}, \cite[Example 2.12]{KL}). 
	As the action of $S_n$ on the equivariant line bundles $\cV_i/\cV_{i-1}$ permutes the equivariant weights, we find that the $T$-equivariant cohomology of $X$ is 
	$$H^*_T(X)\cong H^*(X)\otimes \QQ[t_1,\cdots,t_n]$$
	where $S_n$ acts trivially on $H^*(X)$ and by \eqref{n13} on $H^*_T=\QQ[t_1,\cdots,t_n]$.
	
	By Remark \ref{n14} and \eqref{15}, we find that 
	\beq\label{n17}
	H^*(Y)\otimes \QQ[t_1,\cdots,t_n]\cong H^*_T(Y)\cong H^*_T(X)\cong H^*(X)\otimes \QQ[t_1,\cdots,t_n] \eeq
	is $S_n$-equivariant where $S_n$ acts trivially on the left $\QQ[t_1,\cdots,t_n]$ and by \eqref{n13} on the right $\QQ[t_1,\cdots,t_n]$. 
	
	Applying the Frobenius characteristic to \eqref{n17}, we obtain 
	\beq\label{n15}
	\cK_n:=\ch_q(H^{2*}(Y))=(1-q)^n[n]_q! \ch (\QQ[t_1,\cdots,t_n])~\in \Lambda[q]\eeq
	because 
	$\ch_q(H^*(X))=[n]_q!$.
	From this, it follows that $\cK_n$ is uniquely determined by the inductive formula \eqref{13}. 
	Indeed, $f_n:=\ch_q (\QQ[t_1,\cdots,t_n])$ in \eqref{n15} satisfies 
		\[f_n=\frac{1}{1-q^n}\left(e_{1}f_{n-1}-e_2f_{n-2}+\cdots +(-1)^{n-1}e_nf_0\right), \qquad f_0:=1\]
	by Lemma~\ref{n21} (2) below. 
		\end{example}
		
Example \ref{33} immediately implies the following.
\begin{proposition}\label{n16}
The function $h\mapsto \cP(h)$ satisfies  Theorem~\ref{46} (1). 
\end{proposition}

The remaining two conditions in Theorem~\ref{46} will be proved in \S\ref{n18} and in \S\ref{S4}.  

\medskip

\begin{lemma}\label{n21}
	Let $f_n:=\ch(\QQ[t_1,\cdots,t_n])$ in \eqref{n15} with $f_0=1$. Then $\{f_n\}_{n\geq 0}$ satisfies the following: 
	\begin{enumerate}
		\item $ f_n=\sum_{i=0}^n q^{n-i}h_if_{n-i}$ ~ and 
		\item $q^nf_n=\sum_{i=0}^n(-1)^ie_if_{n-i}$
	\end{enumerate}
	where 
	$h_m=\sum_{i_1\leq \cdots \leq i_m}x_{i_1}\cdots x_{i_m} \in \Lambda$ is the $m$-th complete homogeneous symmetric function with $h_0=1$. 
\end{lemma}

\begin{proof}
	(1) Let $W_n$ (resp. $W_n'$) be the submodule in $\QQ[t_1,\cdots,t_n]$ in \eqref{n15} spanned by monomials generated by less than (resp. precisely) $n$ variables. Then, we have
	\[\QQ[t_1,\cdots,t_n]=\bigoplus_{k\geq0}(t_1\cdots t_n)^kW_n=W_n\oplus W_n'\] 
	Hence $f_n=\frac{1}{1-q^n}\ch_q (W_n)=\ch_q(W_n)+\ch_q(W_n')$. In particular,
	\beq \label{n23} \ch_q(W_n)=(1-q^n)f_n \and \ch_q(W_n')=q^nf_n.\eeq
	Moreover, $W_n$ admits a decomposition by the number of generating variables
	\[W_n=\bigoplus_{i=1}^n\Ind^{S_n}_{S_i\times S_{n-i}}W_{n-i}'\]
	where $\Ind^{S_n}_{S_i\times S_{n-i}}W'_{n-i}$ is the induced representation of $W_{n-i}'$ as an $S_i\times S_{n-i}$ representation with a trivial $S_i$-action. Since $\ch$ is multiplicative with respect to the multiplication on representations of symmetric groups given by taking the induced representation of tensor products (\cite[I,~(7.3)]{Mac}), we have
	\[(1-q^n)f_n=\ch_q(W_n)=\sum_{i=1}^nh_i\ch_q(W_{n-i}')=\sum_{i=1}^n q^{n-i}h_if_{n-i}\]
	by \eqref{n23}. This proves (1).
	
	(2) We use (1) and a well-known identity (\cite[I,~(2.6)]{Mac})
	\beq\label{n24} h_n=(-1)^{n-1}e_n +\sum_{j=1}^{n-1}(-1)^jh_{n-j}e_j \quad \text{ for }n\geq2 \and h_1=e_1,
	\eeq
	together with induction on $n$. 
	
	When $n=0$, (2) trivially holds since $q^0f_0=1=e_0f_0$. Assume that $n\geq 1$ and (2) holds for every $m<n$, so that
	\beq \label{n26} q^mf_m=f_m+\sum_{j=1}^m(-1)^je_jf_{m-j}.\eeq
	Then, by the assertion (1), \eqref{n26} and \eqref{n24}, we have
	\[\begin{split}
		&(1-q^n)f_n=\sum_{i=1}^n(q^{n-i}f_{n-i})h_i=\sum_{i=1}^nh_if_{n-i}+\sum_{i=1}^n\sum_{j=1}^{n-i}(-1)^jh_ie_jf_{n-i-j}\\
		&=\sum_{i=1}^n(-1)^{i-1}e_if_{n-i} +\sum_{i=2}^n\sum_{j=1}^{i-1}(-1)^{j-1}h_{i-j}e_jf_{n-i}+\sum_{i=1}^n\sum_{j=1}^{n-i}(-1)^jh_ie_jf_{n-i-j}\\
		&=\sum_{i=1}^n(-1)^{i-1}e_if_{n-i}
	\end{split}\]
	where the last equality holds by 
	\[\sum_{i=2}^n\sum_{j=1}^{i-1}(-)=\sum_{k,j\geq 1, k+j\leq n}(-)=\sum_{k=1}^n\sum_{j=1}^{n-k}(-)\]
	with $i=j+k$. This proves the assertion (2).
\end{proof}

A by-product of Lemma~\ref{n21} is an elementary proof of the following. 

\begin{corollary} \cite[Lemma~5.0.1~(2)]{MS}\label{n27}
Let $\w:\Lambda\cong \Lambda$ be the involution of the ring $\Lambda$ which interchanges $e_n$ and $h_n$. 
The we have 
	$$f_n(q^{-1})=(-q)^n (\w f_n)(q).$$
\end{corollary}
\begin{proof}
	Let $f'_n:=\w f_n$ for $n\geq0$. We use induction on $n$. The assertion trivially holds for $n=0$. 
	Assume that $n\geq 1$ and $f_m(q^{-1})=(-q)^mf'_m(q)$ for $m<n$.
	By applying $\w$ to Lemma~\ref{n21}~(2), we have
	\[\begin{split}
		(1-q^n)\w f_n(q)&=(1-q^n)f_n'(q)=\sum_{i=1}^n(-1)^{i-1}h_if_{n-i}'(q)\\
		&=(-1)^{n-1}\sum_{i=1}^nq^{-(n-i)}h_if_{n-i}(q^{-1})\\
		&=(-1)^{n-1}(1-q^{-n})f_n(q^{-1})\\
		&=(-q)^{-n}(1-q^n)f_n(q^{-1})
	\end{split}\]
	where the third and the fourth equalities hold by the induction hypothesis and Lemma~\ref{n21}~(1) respectively. Therefore the corollary holds for all $n$.
	\end{proof}

Corollary \ref{n27} is a key ingredient in the proof of the parlindromicity of $\cP(h)$ up to the involution $\w$ by Masuda and Sato in \cite{MS} which says that 
\beq\label{r2} \begin{split}
	&H^*(Y_h)\cong H^{\dim_\RR Y_h-2*}(Y_h)\otimes \mathrm{sgn} \quad \text{ or equivalently, }\\
	&\ch_q(H^{2*}(Y_h))=\w\,\ch_q(H^{\dim_\RR Y_h-2*}(Y_h))
\end{split}\eeq
where $\mathrm{sgn}$ denotes the sign representation.
In fact, this parlindromicity \eqref{r2} follows from Corollary~\ref{n27} and the modular law (Theorem \ref{2}) as follows: First one can immediately check \eqref{r2} for the isospectral manifold $Y$ in \eqref{n7} using Corollary~\ref{n27} and \eqref{n15}. Next the modular law enables us to deduce \eqref{r2} for $Y_h$ from the palindromicity for $Y$.

\begin{remark}
	Complete homogeneous symmetric functions $h_n$ are used only in Lemma~\ref{n21} and the proof of Corollary~\ref{n27}. We hope these not to be confused with Hessenberg functions.
\end{remark}
	
	\bigskip

\subsection{Connectedness and multiplicativity}\label{n18}
In this subsection, we give a criterion for connectedness of twin manifolds, and check the multiplicative property of the function $\cP$ in Definition \ref{30}. 
\begin{lemma}
	For a Hessenberg function $h:[n]\to [n]$ with $n\ge 2$, the twin manifold $Y_h$ is connected if and only if $h(i)>i$ for all $1\leq i<n$. If $h(j)=j$ for some $j<n$, then $Y_h$ is the disjoint union of $\binom{n}{j}$ copies of $Y_{h'}\times Y_{h''}$ for $h'$ and $h''$ defined as in Theorem~\ref{46} (2). 
\end{lemma} 
\begin{proof}
	This is given by the explicit diffeomorphism 
	\beq \label{49} Y_h\cong \bigsqcup_{I\subset[n], ~\lvert I\rvert=j} Y_{h',I}\times Y_{h'', I^c}, \quad \begin{pmatrix}
		A' &O\\ O&A''
	\end{pmatrix}\leftrightarrow (A',A'')\eeq
	where $Y_{h',I}\cong Y_{h'}$ denotes the twin manifold with spectrum $\{\lambda_i\,|\,i\in I\}$ and $Y_{h'',I^c}\cong Y_{h''}$ is defined similarly. 
\end{proof}
\begin{proposition}\label{58}
The function $h\mapsto \cP(h)$ satisfies Theorem~\ref{46} (2).
\end{proposition}
\begin{proof}
	Under the identifications \eqref{25} and \eqref{n10}, the isomorphism \eqref{49} restricts to an isomorphism on $T$-fixed point sets
	\[Y_h^T\cong S_n =\bigsqcup_{I\subset[n],~\lvert I\rvert =j}(S_I\times S_{I^c})w_I\cong \bigsqcup_{I\subset [n], ~\lvert I\rvert =j}Y_{h',I}^T\times Y_{h'',I^c}^T\]
	where $w_I$ is an element of $S_n$ sending $[j]$ and $[j]^c$ to $I$ and $I^c$ respectively such that $w_I|_{[j]}$ and $w_I|_{[j]^c}$ are increasing. 
	This induces an isomorphism 
	\[\bigoplus_{I\subset [n],~ \lvert I\rvert=j}H^*_T(Y_{h',I}^T)\otimes_{H^*_T}H^*_T(Y_{h'',{I^c}}^T)\xrightarrow{~\cong~}H^*_T(Y_{h}^T)\]
	sending $(p_{v_I})_{v_I\in S_I}\otimes (p_{v_{I^c}})_{v_{I^c}\in S_{I^c}}$ to $(p_v)_{v\in S_n}$ with $p_v=p_{v_I}p_{v_{I^c}}$ 
	if 
	$v=v_Iv_{I^c}w_I$
	and $p_v=0$ otherwise for each $I$. This is equivariant under the induced actions of $S_I\times S_{I^c}$ and $S_n$ from \eqref{1} 
	for each component.
	Furthermore, one can easily see that the above isomorphism restricts to an isomorphism
	\[\bigoplus_{I\subset [n],~ \lvert I\rvert=j}H^*_T(Y_{h',I})\otimes_{H^*_T}H^*_T(Y_{h'',{I^c}})\xrightarrow{~\cong~}H^*_T(Y_{h})\]
	which exhibits $H^*_T(Y_h)$ as the induced representation of the $S_j\times S_{n-j}$-representation $H^*_T(Y_{h'})\otimes_{H^*_T} H^*_T(Y_{h''})$. Taking quotients by the submodules generated by $\mathfrak{m}$, $H^*(Y_h)$ is the induced representation of the $S_j\times S_{n-j}$-representation $H^*(Y_{h'})\otimes H^*(Y_{h''})$. 
	Since the Frobenius characteristic $\ch$ is multiplicative with respect to 
	tensor products (cf. \cite[I,~(7.3)]{Mac}),
	we have $$\cP(h)=\cP(h')\cP(h'')$$ as desired. 
\end{proof}

By comparing the twin manifolds geometrically, 
we will prove below that the last condition (3) in Theorem \ref{46} is also satisfied for $\cP(h)$ (cf.~Theorem \ref{2} below) and hence give a direct proof of the LLT-SW conjecture (cf.~Theorem \ref{48}). 

\bigskip

\section{Geometry of  twin manifolds}\label{3}
In this section, we compare the twin manifolds associated to a modular triple $\bh=(h_-,h,h_+)$ of Hessenberg functions in Definition \ref{59}. 
More precisely, we construct a manifold $\tY_\bh$, called the \emph{roof manifold} of $\bh$, together with maps
$$\xymatrix{&\tY_\bh \ar[ld]_-{\pi}\ar[rd]^-{\pr_2}& \\ Y_{h_+}&&\PP^1}$$
where $\pr_2$ is a smooth fibration with fiber $Y_h$ and $\pi$ is the \emph{blowup} along the submanifold $Y_{h_-}$ of complex codimension 2. 
The modular law  
\beq\label{n30}
H^{2k}(Y_h)\oplus H^{2k-2}(Y_h)\cong H^{2k}(Y_{h_+})\oplus H^{2k-2}(Y_{h_-}).\eeq
for $\cP(h)$ then follows immediately from the \emph{blowup formula} for $\pi$  and the spectral sequence for $\pr_2$.  

\medskip

\subsection{Roof of a triple} 

In this subsection, we define the roof $\tY_\bh$ of a triple $\bh=(h_-,h,h_+)$. 

For $I=[a,b]=\{a,a+1,\cdots,b\} \subset [n]$, let
\[\iota_I:U(b-a+1)\lra U(n)\]
be the embedding
\[A~\mapsto ~\begin{pmatrix}
	I_{a-1} &&\\&A&\\&&I_{n-b}
\end{pmatrix}\]
where $I_k$ denotes the $k\times k$ identity matrix for each $k$. When $a=b$, let $\iota_a:=\iota_I$. For disjoint $I=[a,b]$ and $J=[c,d]\subset [n]$, we denote by
\[\iota_I\times \iota_J:U(b-a+1)\times U(d-c+1)\lra U(n)\]
the map $(A,B)\mapsto \iota_I(A) \iota_J(B)$ given by the matrix multiplication.

\begin{definition}\label{43}
	Given a Hessenberg function $h:[n]\to [n]$, consider a triple $\bh=(h_-,h,h_+)$ of Hessenberg functions 
	defined by either of the following.
	\begin{enumerate}
		\item When $h(j)=h(j+1)$ and $h^{-1}(j)=\{j_0\}$ for some $1\leq j_0<j<n$, let $1\leq r\leq n-j $ be any integer such that
		\[h(j)=\cdots =h(j+r) \and h^{-1}(\{j+1,\cdots, j+r-1\})=\emptyset.\] 
		Let $h_-, h_+:[n]\to [n]$ be defined by
		\[h_-(i)=\begin{cases}
			j-1 & \text{for }i=j_0\\
			h(i) &\text{otherwise}
		\end{cases} 
		\and h_+(i)=\begin{cases}
			j+r &\text{for }i=j_0\\
			h(i) & \text{otherwise.}
		\end{cases} \]
		In this case, we define
		\begin{equation*}
			\begin{split}
				\tY_\bh :=Y_h\times _{U(1)\times U(r)}U(r+1)& \quad \text{ and}\\
				E_\bh:=Y_{h_-}\times_{U(1)\times U(r)}U(r+1)&
			\end{split}
		\end{equation*}
		to be the quotients of $Y_{h}\times U(r+1)$ and $Y_{h_-}\times U(r+1)$ by the free actions of $U(1)\times U(r)$ given by
		\beq \label{36} (Tg,A).B=(Tg\iota(B),\iota'(B)^{-1}A)\eeq 
		for $A\in U(r+1)$ and $B\in U(1)\times U(r)$,
		via 
		\begin{equation*}
			\begin{split}
				&\iota=\iota_j\times \iota_{[j+1,j+r]}:U(1)\times U(r)\hookrightarrow U(n)\quad \text{ and}\\
				&\iota'=\iota_1\times \iota_{[2,r+1]}:U(1)\times U(r)\hookrightarrow U(r+1).
			\end{split}
		\end{equation*}
		\item When $h(j)+1=h(j+1)\neq j+1$ and $h^{-1}(j)=\emptyset$ for some $j<n$, let $1\leq r\leq j$ be any integer such that 
		\[ h(j-r+1)=\cdots =h(j) \and h^{-1}(\{j-r+1,\cdots, j\}=\emptyset.\]
		Let $h_-,h_+ :[n]\to [n]$ be defined by 
		\[h_-(i)=\begin{cases}
			h(j) &\text{for }i=j+1\\
			h(i) &\text{otherwise}
		\end{cases} 
		\and h_+(i)=\begin{cases}
			h(j)+1 &\text{for }j-r<i\leq j\\
		h(i) &\text{otherwise.}
		\end{cases}\]
		In this case, similarly we define 
		\begin{equation*}
			\begin{split}
				\tY_\bh :=Y_h\times_{U(r)\times U(1)}U(r+1)& \quad \text{ and}\\
				E_\bh:=Y_{h_-}\times_{U(r)\times U(1)}U(r+1)&
			\end{split}
		\end{equation*}
		given by the actions \eqref{36} for $B\in U(r)\times U(1)$,
		via 
		\begin{equation*}
			\begin{split}
				&\iota=\iota_{[j-r+1,j]}\times \iota_{j+1}:U(r)\times U(1)\hookrightarrow U(n) \quad \text{ and}\\
				&\iota'=\iota_{[1,r]}\times \iota_{\{r+1\}}:U(r)\times U(1)\hookrightarrow U(r+1).
			\end{split}
		\end{equation*} 
	\end{enumerate}
	We call a triple $\bh=(h_-,h,h_+)$ in (1) and (2) a triple of \emph{type (1)} and \emph{(2)} respectively.
\end{definition}
\begin{remark}\label{r1}
Note that when $r=1$, $\bh$ is a modular triple in Definition~\ref{59}.
Also note that triples of type (1) and (2) with the same $r$ are dual to each other via the involution map $h\mapsto h^t$ on the set of Hessenberg functions, where for a Hessenberg function $h:[n]\to[n]$, its transpose $h^t:[n]\to [n]$ is defined by
\[h^t(i)=n-i_{\max}, \quad  i_{\max}=\max\{j\in [n]:h(j)<n+1-i\}.\]
One can easily see that $h$ is a triple of type (1) (resp. (2)) if and only if $h^t$ is a triple of type (2) (resp. (1)). 
\end{remark}

\begin{example}\label{8}
	Let $h=(2,5,5,5,6,6)$ where we write $h=(h(1),\cdots, h(n))$.
	Then $\bh=(h_-,h,h_+)$ with 
	\[h_-=(1,5,5,5,6,6) \and h_+=(4,5,5,5,6,6)\]
	 is a triple of type (1) with $(j_0,j,r)=(1,2,2)$. Similarly, $\bh=(h_-,h,h_+)$ with 
	 \[h_-=(2,5,5,5,5,6) \and h_+=(2,5,6,6,6,6)\]
	 is a triple of type (2) with $(j,r)=(4,2)$.	 
\end{example}

\medskip

\subsection{Maps from the roof}
In this subsection, we show that for a triple $\bh=(h_-,h,h_+)$ in Definition \ref{43}, \begin{enumerate}
\item[(* )] $\tY_\bh$ is the blowup of $Y_{h_+}$ along $Y_{h_-}$ with exceptional divisor $E_\bh$  and 
\item[(**)] $\tY_\bh$ is a fiber bundle over $\PP^r$ with fiber $Y_h$.
\end{enumerate} 
These are our \emph{geometric relations} among the twin manifolds  
that will give us the modular law \eqref{n30}. 

\medskip

To prove (*) and (**), we have to construct maps among the manifolds $\tY_\bh$, $E_\bh$, $Y_{h_+}$, $Y_{h_-}$ and $\PP^r$.  
As \beq\label{n25}
 \left(U(1)\times U(r)\right)\backslash U(r+1)\cong \PP^r \cong \left(U(r)\times U(1)\right)\backslash U(r+1),\eeq 
the second projection 
$$Y_h\times U(r+1)\lra U(r+1), \quad (\text{resp. } Y_{h_-}\times U(r+1)\lra U(r+1) )$$ 
induces the fiber bundle  
\beq \label{9} \pr_2:\tY_{\bh}\lra \PP^r, \quad (\text{resp. } E_\bh\lra \PP^r )\eeq
whose fibers are $Y_h$ (resp. $Y_{h_-}$). The obvious inclusion $Y_{h_-}\subset Y_h$ induces the inclusion 
	\beq\label{n19}\jmath:E_\bh\hookrightarrow \tY_\bh\eeq
of fiber bundles over $\PP^r$. 
Moreover, we have the map  
	\beq\label{n20} \pi:\tY_\bh \lra Y_{h_+}, \quad [Tg,A]\mapsto Tg\iota(A)\eeq 
	for $\iota=\iota_{[j,j+r]}$ (resp. $\iota=\iota_{[j-r+1,j+1]}$) if $\bh$ is of type (1) (resp. type (2)). Similarly, we have the map  
	\beq\label{n22} \pi_-: E_\bh\lra Y_{h_-}, \quad [Tg,A]\mapsto Tg\iota(A).\eeq 
	These are well defined, since $Y_{h_+}$ and $Y_{h_-}$ are invariant under the right multiplication of $\iota(U(r+1))$ by \eqref{25}.

Using the notation of \eqref{23}, let $f_\bh:Y_{h_+}\to \C^{r+1}$ be a map defined by 
	\begin{equation*}
		f_\bh(y)=\begin{cases}
			(f_{j_0,j}(y),~\cdots ,~ f_{j_0,j+r}(y)) &\text{for type (1)}\\
			(f_{h(j+1),j-r+1}(y),~\cdots ,~ f_{h(j+1),j+1}(y)) &\text{for type (2)}
		\end{cases}
	\end{equation*}
	so that 
	\beq \label{12} Y_{h_-}=\{y\in Y_{h_+}:f_\bh(y)=0\}\eeq
	is the transversal vanishing locus of $f_\bh$ 
	by \eqref{25} and \eqref{23}. This induces a map 
	\beq  \label{11} Y_{h_+}-Y_{h_-}\lra  \PP^r, \qquad y\mapsto [f_\bh(y)].\eeq

\begin{example}
	Let $\bh$ be as in Example~\ref{8} (1). Then, we have 
	\begin{equation*}
		\begin{split}
			Y_{h}&=\{y\in Y_{h_+}:f_{13}(y)=f_{14}(y)=0\} \quad \text{ and}\\
			Y_{h_-}&=\{y\in Y_{h_+}:f_{12}(y)=f_{13}(y)=f_{14}(y)=0\}.
		\end{split}
	\end{equation*}
	The map \eqref{11} is given by
	\[Y_{h_+}-Y_{h_-}\lra \PP^{2}, \quad y\mapsto [f_{12}(y):f_{13}(y):f_{14}(y)].\]
	\end{example}

The following propositions illustrate the geometry of  $Y_h$ for a triple $\bh$ which is very similar to that of a triple of Hessenberg varieties $X_h$ studied in \cite[\S3.3]{KL}, via blowups and projective bundles. 

\begin{proposition}\label{40}
	Let $\bh=(h_-,h,h_+)$ be a triple in Definition~\ref{43}. Then, \eqref{9} and \eqref{n22} give us a diffeomorphism
	\beq \label{44} 
		(\pi_-,\pr_2):E_\bh \xrightarrow{~\cong~} Y_{h_-}\times \PP^r.\eeq
\end{proposition}
\begin{proof}
	Note that $Y_{h_-}$ is invariant under the action of $\iota(U(r+1))$. Hence the map $(Tg,[A])\mapsto [Tg\iota (A)^{-1},A]$ is the well defined inverse.
	\end{proof}
\begin{proposition} \label{38}
	The map
		\beq \label{41} (\pi,\pr_2):\tY_\bh \lra  Y_{h_+}\times \PP^r\eeq
		induced by \eqref{9} and \eqref{n20} is an embedding of $\tY_\bh$ onto the submanifold defined by 
		\beq \label{42} \tY_\bh \cong \{(y,[v])\in Y_{h_+}\times \PP^r~:~ f_{\bh}(y) \in \C v\}.\eeq
		In particular, \eqref{11} fits into the diagram   
		\beq \label{37}  
		\xymatrix{&\tY_\bh \ar[ld]_-{\pi}\ar[rd]^-{\pr_2}& \\ Y_{h_+}\ar@{-->}[rr]^{\eqref{11}}&&\PP^r}\eeq
		where $\pr_2$ is a smooth fibration with fiber $Y_h$.
		\eqref{11}
\end{proposition}
\begin{proof}
	Since $Y_{h_+}$ is invariant under the action of $\iota(U(r+1))$ for $\iota$ defined in Definition~\ref{43}, 
	the same argument in the proof in Proposition~\ref{40} proves that $Y_{h_+}\times \PP^r$ is isomorphic to the quotients $Y_{h_+}\times_{U(1)\times U(r)}U(r+1)$ or $Y_{h_+}\times_{U(r)\times U(1)}U(r+1)$ defined in the same manner as in Definition~\ref{43} (1) and (2) respectively.
 	Under this isomorphism, \eqref{41} is induced from the canonical inclusion $Y_h\subset Y_{h_+}$, in particular it is an embedding. 
 	
 	To see \eqref{42}, note that a point $(y,a)=(Tg,[A])\in Y_{h_+}\times \PP^r$ with $A\in U(r+1)$ lies in the image of \eqref{41} if and only if $Tg\iota(A)^{-1}\in Y_h$, or equivalently, $\iota(A)(g^{-1}x g)\iota(A)^{-1}$ is contained in \eqref{25}. 
  	The latter is equivalent to that the last (resp. first) $r$ coordinates of the 
  	$j_0$-th (resp. $h(j+1)$-th)
  	column vector
 	\begin{equation}\label{52}
 		\begin{split}
 			&A (f_{j,j_0}(y),~\cdots,~f_{j+r,j_0}(y))^t\qquad =A \overline{f_\bh(y)}^t\\
 			(\text{resp. } ~&A (f_{j-r+1,h(j+1)}(y),~\cdots,~ f_{j+1,h(j+1)} )^t)
 		\end{split}
 	\end{equation}
 	of $\iota(A)(g^{-1}x g)\iota(A)^{-1}$ vanish if $\bh$ is of type (1) (resp. type (2)), by \eqref{25} and \eqref{23}. 
 	Here $f_{i,k}(y)$ denotes the $(i,k)$-th component of the matrix $g^{-1}xg$ by \eqref{n7} and \eqref{23}. Also note that the right multiplication of $\iota(A)^{-1}$ in $\iota(A)(g^{-1}x g)\iota(A)^{-1}$ does not change the $j_0$-th (resp. $h(j+1)$-th) column vector since we assume $j_0<j$ (resp. $h(j+1)>j+1$) in Definition~\ref{43}.

 	Under the isomorphisms \eqref{n25} which send the $a=[A]$ to the class represented by the first (resp. last) row vectors of $A$, this is equivalent to that the vector $f_\bh(y)$ is parallel to a vector $v$ representing $a=[v]$, or equivalently, $f_\bh(y)\in \C v$. 
 	The last assertion is immediate.
\end{proof}

\begin{proposition}\label{39} ~
\begin{enumerate}
	\item The left square in  the commutative diagram
	\[\xymatrix{E_\bh\ar@{^(->}[r]^-{\jmath} \ar[d]_-{\pi_-}&\tY_\bh  \ar@{^(->}[r] \ar[d]_-{\pi} & Y_{h_+}\times \PP^r \ar[ld] \\ Y_{h_-} \ar@{^(->}[r] & Y_{h_+}&}\]
		is Cartesian.   		
		\item  $\pi$ is a diffeomorphism over $Y_{h_+}-Y_{h_-}$ with the inverse map 
		\beq \label{51} Y_{h_+}-Y_{h_-}\lra \tY_\bh\subset  Y_{h_+}\times \PP^r, \qquad y\mapsto (y,[f_\bh(y)]).\eeq 
		\item $\pi$ is the trivial $\PP^r$-bundle over $Y_{h_-}$ via \eqref{44}.
		\item The normal bundle of $\jmath$ is isomorphic to the pullback of the tautological complex line bundle $\sO_{\PP^r}(-1)$ via \eqref{44}, as a real vector bundle. 
		\end{enumerate}
\end{proposition}
\begin{proof} (1) and (3) 
	follow from \eqref{12}, which implies the vanishing of \eqref{52} for every $A\in U(r+1)$, so that $\pi^{-1}(Y_{h_-})=Y_{h_-}\times \PP^r\cong E_\bh$. Furthermore, (2) follows from Proposition~\ref{38}. 
	(4) follows from the fact that $E_\bh$ is the vanishing locus in $\tY_\bh$ of the map
	\[(f_{j_0,j},\mathrm{id}):\tY_{\bh}=Y_h\times_{U(1)\times U(r)}U(r+1)\lra\C\times _{U(1)\times U(r)}U(r+1)\]
	when $\bh$ is of type (1) and
	\[(f_{h(j+1),j+1},\mathrm{id}):\tY_{\bh}=Y_h\times_{U(r)\times U(1)}U(r+1)\lra\C\times _{U(r)\times U(1)}U(r+1)\]
	when $\bh$ is of type (2) respectively, and the fact that $E_\bh\subset \tY_\bh$ is submanifold of real codimension two. The complex line bundles $\C\times _{U(1)\times U(r)}U(r+1)$ and $\C\times _{U(r)\times U(1)}U(r+1)$ are the tautological complex line bundle $\sO_{\PP^r}(-1)$ over $\PP^r$ by \eqref{24}, via the isomorphisms \eqref{n25}.
\end{proof}
\bigskip

\subsection{Cohomology of the roof}
In this subsection, we compare the cohomology and $T$-equivariant cohomology of  
the  twin manifolds $Y_{h_-}$, $Y_{h}$ and $Y_{h_+}$ associated to a triple $\bh=(h_-,h,h_+)$  
 by Propositions \ref{38} and \ref{39}.

First observe that we have natural $T$-actions on $\tY_\bh$ and $E_\bh$ as follows. 
\begin{definition}
	Let $\bh=(h_-,h,h_+)$ be a triple in Definition~\ref{43}. 
	Define (right) $T$-actions on $\tY_\bh$ and $E_\bh$ by  
	\beq \label{45} [Tg,A].t=[Tgt,(t')^{-1}At']\eeq
	for $t=\diag(t_1,\cdots, t_n)\in T$ and $t'\in U(r+1)$ given by
	\[t'=\begin{cases}
		\diag(t_j,\cdots, t_{j+r}) &\text{ if $\bh$ is of type (1),}\\
		\diag(t_{j-r+1},\cdots, t_{j+1})  &\text{ if $\bh$ is of type (2).}
	\end{cases}\]
	This induces a natural $T$-action on $\PP^r$ via \eqref{n25}, 
	which coincides with
	the componentwise multiplication of $t'$.
\end{definition}
 It is straightforward to see that all the morphisms in Propositions~\ref{40}, \ref{38} and \ref{39} are $T$-equivariant.

\medskip

By Proposition \ref{38}, $\pr_2$ is a fiber bundle with fiber $Y_h$. 
By Proposition \ref{n32} and \eqref{n31}, the odd degree parts of $H^*(Y_h)$ and $H^*_T(Y_h)$ vanish. Hence the spectral sequence for $\pr_2$ degenerates and we have isomorphisms 
\beq\label{n33}
H^*(\tY_\bh)\cong H^*(Y_{h})\otimes H^*(\PP^r) \and H^*(E_\bh)\cong H^*(Y_{h_-})\otimes H^*(\PP^r).
\eeq
Letting $\gamma=c_1(\sO_{\PP^r}(-1))$, we have a ring isomorphism
$H^*(\PP^r)\cong \QQ[\gamma]/(\gamma^{r+1})$ and the second isomorphism in \eqref{n33} is in fact the inverse of 
\beq\label{r3} \sum_i\beta_i\otimes \gamma^i\mapsto 
\sum_i\gamma^i\cup \pi_-^*\beta_i,  
 \quad \beta_i\in H^{*-2i}(Y_{h_-})\eeq 
by Proposition \ref{40} above.  Here we are abusing the notation by denoting the pullback of $\gamma$ to $E_\bh$ by $\gamma$ to simplify the notation.

Similarly as in \eqref{n33}, we have
\beq\label{n34}
H^*_T(\tY_\bh)\cong H^*_T(Y_{h})\otimes H^*(\PP^r) \and H^*_T(E_\bh)\cong H^*_T(Y_{h_-})\otimes H^*(\PP^r)
\eeq
for the $T$-equivariant cohomology.

By Proposition~\ref{39}, we have the following 
 \emph{blowup formula}.
\begin{proposition} \label{16} Let $\bh=(h_-,h,h_+)$ be as above. Then the map
	\beq \label{17} H^*(Y_{h_+}) \oplus \bigoplus_{i=1}^{r}H^{*-2i}(Y_{h_-})\xrightarrow{~\cong~}H^*(\tY_{\bh})\eeq
	sending $(\a,\b_1,\cdots,\b_{r})$ to  $\pi^*\a+
	\sum_{i=1}^{r} \jmath_*(e(N_\jmath)^{i-1}\cup \pi_-^*\b_i)$
	is an isomorphism,
	where $e(N_\jmath)$ denotes the Euler class of the normal bundle $N_\jmath$ of the canonical inclusion $\jmath:E_{\bh}\hookrightarrow \tY_{\bh}$ 
	and $\jmath_*$ denotes the Gysin homomorphism induced by $\jmath$. 
	Similarly, the map
	\beq \label{18} H^*_T(Y_{h_+}) \oplus \bigoplus_{i=1}^{r}H^{*-2i}_T(Y_{h_-})\xrightarrow{~\cong~}H^*_T(\tY_{\bh})\eeq
	sending $(\a,\b_1,\cdots,\b_{r})$ to  $\pi^*\a+
	\sum_{i=1}^{r} \jmath_*(e^T(N_\jmath)^{i-1}\cup \pi_-^*\b_i)
	$
	is an $H^*_T$-module isomorphism, where $e^T(N_\jmath)$ denotes the $T$-equivariant Euler class of $N_\jmath$. 
\end{proposition}
\begin{proof}
	We will show that $\jmath^*:H^*(\tY_\bh) \to H^*(E_\bh) $ induces an isomorphism
	\beq\label{47}
	\mathrm{Coker}\left(\pi^*\right)\xrightarrow{~\cong~}\mathrm{Coker}\left(\pi_-^*\right)\cong \bigoplus_{i=1}^rH^{*-2i}(Y_{h_-})\otimes H^{2i}(\PP^r)
	\eeq 
	using Propositions~\ref{40} and \ref{39}.
	Then a splitting of a short exact sequence
	\[0\lra H^*(Y_{h_+})\xrightarrow{\;\pi^*\;} H^*(\tY_\bh)\lra \bigoplus_{i=1}^rH^{*-2i}(Y_{h_-})\lra 0\]
	is given by 	
	the map \[\bigoplus_{i=1}^rH^{*-2i}(Y_{h_-})\lra  H^*(\tY_\bh), \quad (\b_1,\cdots,\b_r)\mapsto \sum_{i=1}^r \jmath_*\left(\gamma^{i-1}\cup \pi_-^*\b_i\right)\]
	since 
	$\jmath^*\jmath_*\left(\gamma^{i-1}\cup \pi_-^*\b_i\right)=\gamma^{i}\cup\pi_-^*\b_i$
	by Proposition~\ref{39} (4) and this corresponds to $\b_i$ via \eqref{r3}.

	In the rest of the proof, we show that \eqref{47} is an isomorphism.
	By \eqref{n33},  $H^{2k+1}(\tY_\bh)=H^{2k+1}(E_\bh)=0$ for all $k$. 
	Similarly, we have $H^{2k+1}_T(\tY_\bh)=H^{2k+1}_T(E_\bh)=0$ for all $k$.
	From (1) 
	and (2)
	in Proposition~\ref{39},
	we have a commutative diagram of exact sequences 
	\small
\beq \label{7}\xymatrix{ 
&&0\ar[d] &0\ar[d]\\
0\ar[r] &H^{2k}_c(Y_{h_+}-Y_{h_-})\ar[r] \ar[d]^-{\cong }& H^{2k}(Y_{h_+}) \ar[r]\ar[d]^-{\pi^*} & H^{2k}(Y_{h_-}) \ar[r]\ar[d]^-{\pi_-^*}& H^{2k+1}_c(Y_{h_+}-Y_{h_-})\ar[d]^-{\cong }\ar[r] &0 \\ 
0\ar[r] &H^{2k}_c(\tY_{\bh}-E_{\bh}) \ar[r] & H^{2k}(\tY_{\bh}) \ar[r]^-{\jmath^*} \ar[d] & H^{2k}(E_{\bh}) \ar[r]\ar[d] & H^{2k+1}_c(\tY_{\bh}-E_{\bh})\ar[r] &0\\ 
&& \mathrm{Coker}(\pi^*)  \ar[r]^-{\cong }\ar[d] & \mathrm{Coker}(\pi_-^*)\ar[d]\\ 
&& 0 & 0
}\eeq
\normalsize
without two 0's at the top,
for each $k$,  where the two rows are parts of the long exact sequences of cohomology with compact supports. 
By \eqref{44} and the five lemma, $\pi_-^*$ and $\pi^*$ are injective respectively. Then one can check that the horizontal arrow at the bottom is an isomorphism by a diagram chase.
In particular, we have
\[\mathrm{Coker}(\pi^*)\cong \mathrm{Coker}(\pi_-^*)\cong \bigoplus_{i=1}^rH^{2k-2i}(Y_{h_-})\otimes H^{2i}(\PP^r) \] 
in \eqref{47}, where the second isomorphism follows by Proposition~\ref{40}.

By the same argument with the equivariant cohomology instead of the ordinary cohomology, the second assertion also follows, since all maps used in the argument above, including $\pi_-$, $\pi$ and $\jmath$, are $T$-equivariant.
\end{proof}

\medskip

\subsection{$S_n$-representations on the cohomology of the roof}
In this subsection, we prove the following.

\begin{proposition}\label{n35}
There are $S_n$-actions on $H^*_T(\tY_\bh)$ and $H^*(\tY_\bh)$ so that \eqref{n33}, \eqref{n34}, \eqref{17} and \eqref{18} are all isomorphisms of $S_n$-representations. 
\end{proposition}

\begin{proof} 
The statement for the ordinary cohomology follows from that for the equivariant cohomology by \eqref{n31}.  
By using the $S_n$-actions on $H^*_T(Y_h)$, $H^*_T(Y_{h_-})$ and $H^*_T(Y_{h_+})$, we can define two actions of $S_n$ on $H^*_T(\tY_\bh)$ by the isomorphisms \eqref{n34} and \eqref{18}. 
We have to show that the two actions of $S_n$ coincide. 

Note that the set of $T$-fixed points in $\tY_\bh$ is
		\beq \label{26} \tY_{\bh}^T \cong Y_{h_+}^T\times (\PP^r)^T \cong S_n\times \{\sigma_i\}_{0\leq i\leq r}\eeq
		where $\sigma_i$ denote the $i$-th coordinate points.
		Indeed, by \eqref{44} and \eqref{41}, we have
		\[Y_{h_-}^T\times (\PP^r)^T=E_\bh^T~\subset~ \tY_h^T ~\subset ~(Y_{h_+}\times \PP^r)^T=Y_{h_+}^T\times (\PP^{r})^T\]
		where all the inclusions are equalities since $Y_{h_-}^T=Y_{h_+}^T\cong S_n$.
		
By \eqref{n33} and \eqref{n34},  the odd degree part of $H^*(\tY_{\bh})$ vanishes and hence the roof $\tY_\bh$ is equivariantly formal. In particular, the restriction map
		\beq \label{21} \res: H^*_T(\tY_{\bh})\lra H^*_T(\tY_{\bh}^T)\eeq
by the inclusion $\tY_\bh^T\subset \tY_\bh$ is injective.

		Consider a natural $S_n$-action on $H^*_T(\tY_{\bh}^T)$ defined by
		\beq \label{20} (\mu,(p_{v,\sigma})_{(v,\sigma)\in \tY_{\bh}^T})\mapsto  (p_{\mu^{-1}v,\sigma})_{(v,\sigma)\in \tY_{\bh}^T}\eeq
		for $\mu\in S_n$ and $p_{(v,\sigma)}\in H^*_T\cong \Q[t_1,\cdots, t_n]$.
Proposition \ref{n35} then follows if we show that the isomorphisms \eqref{n34} and \eqref{18} composed with \eqref{21} is $S_n$-equivariant with respect to the action \eqref{20}.  
Therefore the proposition follows from Lemmas \ref{n37} and \ref{n38} below. 
\end{proof}

\begin{lemma}\label{n37}
The composition of \eqref{18} and \eqref{21} is $S_n$-equivariant. 
\end{lemma}
\begin{proof}
Consider the commutative diagram 
	\[\xymatrix{\tY_{\bh}^T\ar@{^(->}[r]\ar[d]^-{\pi^T} & E_{\bh}\ar@{^(->}[r]^-{\jmath} \ar[d]^-{\pi_-} & \tY_{\bh} \ar[d]^-{\pi} \\ Y_{h_+}^T \ar@{^(->}[r] & Y_{h_-}\ar@{^(->}[r] & Y_{h_+}}\]
	where $\pi^T$ denotes the restriction of $\pi$ to the 
	fixed point set.
It suffices to show that the embedding of each component of the left hand side of \eqref{18} into $H^*_T(\tY_\bh^T)$ is $S_n$-equivariant with respect to the actions \eqref{1} and \eqref{20}.  Equivalently, it suffices to show that the compositions
	\begin{center}
		$\res\circ \pi^*=(\pi^T)^*\circ \res :H^*_T(Y_{h_+})\lra H^*_T(\tY_{\bh}^T)$\\
		$\res \circ \jmath_*\circ e^T(N_\jmath)^k\circ\pi_-^*=e^T(N_\jmath)^{k+1}|_{\tY_{\bh}^T}\circ (\pi^T)^*\circ \res :H^*_T(Y_{h_-})\lra H^*_T(\tY_{\bh}^T)$
	\end{center}
	are $S_n$-equivariant for $0\leq k<r$.
	
	Since $\res$ in \eqref{21} and $(\pi^T)^*$ are $S_n$-equivariant as $\pi^T(v,\sigma)=v$ under the isomorphism \eqref{26}, it is enough to show that the $T$-equivariant Euler class map $e^T(N_\jmath)|_{\tY_{\bh}^T}: H^*_T(\tY_{\bh}^T)\to H^{*+2}_T(\tY_{\bh}^T)$ is $S_n$-equivariant. Indeed, by Proposition~\ref{39} (2), each $N_\jmath|_{(v,\sigma_i)}=\pr_2^* (\sO_{\PP^r}(-1)|_{\sigma_i})$ with $v\in S_n$ and $0\leq i\leq r$ does not depend on $v$ and hence $e^T(N_\jmath)|_{\tY_{\bh}^T}$ is $S_n$-equivariant.
\end{proof}	
\begin{remark}
	More precisely, the $T$-equivariant Euler class of $\sO_{\PP^r}(-1)|_{\sigma_i}$ above is $t_{j_0}-t_{j+i}$ for type (1) and $t_{j-r+1+i}-t_{h(j+1)}$ for type (2) respectively, up to sign, by the proof of Proposition~\ref{39} (4) and \eqref{24}.
\end{remark}

\begin{example}[$r=1$]
	Let $\bh=(h_-,h,h_+)$ be a modular triple of type (1). In particular, $r=1$. Then, the inclusions
	\[\xymatrix{Y_h\cong \pr_2^{-1}(0)\ar@{^(->}[r]^-{\jmath} & \tY_\bh & \pr_2^{-1}(\infty)\ar@{_(->}[l]_-{\imath}}\]
	induce the short exact sequence
	\[H_T^{2k-2}(\pr_2^{-1}(0))\xrightarrow{~\jmath_*~} H_T^{2k}(\tY_\bh)\xrightarrow{~\imath^*~} H^{2k}_T(\pr_2^{-1}(\PP^1\setminus\{0\})\cong H_T^{2k}(\pr_2^{-1}(\infty))\lra 0\]
	where $\jmath_*$ is indeed injective, since $\jmath^*\jmath_*$ is the multiplication by $e^T(N_\jmath)=\pm (t_{j+1}-t_j)$ which is not a zero divisor. 
	This gives the decomposition
	\beq \label{60} H^*_T(\tY_\bh)\cong  H^{*-2}_T(\pr_2^{-1}(0))\oplus H^*_T(\pr_2^{-1}(\infty)).\eeq

	Under the identification 
	\[\phi: Y_h\cong \pr_2^{-1}(0)\xrightarrow{~\cong~}\pr_2^{-1}(\infty), \quad Tg\mapsto Tg\cdot (j,j+1)\] 
	the isomorphism \eqref{60} now reads as 
	\[H^*_T(\tY_\bh)\cong H^{*-2}_T(Y_h)\oplus H^*_T(Y_h),\]
	which preserves the submodule generated by $\mathfrak{m}$. One can also immediately check that it is $S_n$-equivariant since $e^T(N_\jmath)$ and $\phi^*\circ \imath^*$ are. 
	Hence, we have
	\[H^*(\tY_\bh)\cong H^{*-2}(Y_h)\oplus H^*(Y_h)\]
	which is $S_n$-equivariant.
	\end{example}
	The arguments used in the above example easily extend to a more general setting in the following lemma.

\begin{lemma}\label{n38}
The composition 
$$ \bigoplus_{i=0}^r H^{*-2i}_T(Y_h)\cong H^*_T(Y_h)\otimes H^*(\PP^r)\cong H^*_T(\tY_\bh)\hookrightarrow H^*_T(\tY_\bh^T)$$
of \eqref{n34} and \eqref{21} is $S_n$-equivariant where $S_n$ acts trivially on $H^*(\PP^r)$. Furthermore, the above isomorphism preserves the submodule generated by $\mathfrak{m}$.
\end{lemma}
\begin{proof}
Suppose $\bh$ is of type (1). 
	Let $H_k\cong \PP^{k}$ be the coordinate plane in $\PP^r$ spanned by the coordinate points $\sigma_0,\cdots, \sigma_k$. This induces a $T$-equivariant filtration
	\[Y_h\cong \tY_0 \subset \cdots \subset \tY_r=\tY_\bh\]  	
	of $\tY_\bh$, where $\tY_k$ are given by
	\[\tY_k:=\pr_2^{-1}(H_k)\cong\{(y,[v])\in Y_{h_+}\times H_k:f_\bh(y)\in \C v\}.\]
	Equivalently, $\tY_k$ is the (smooth) intersection of $\tY_\bh$ and $Y_{h_+}\times H_k$ in $Y_{h_+}\times \PP^r$.
	Let $\jmath_k:\tY_{k-1}\subset \tY_{k}$ denote the inclusion. Associated to this filtration, there is a Gysin sequence
\beq\label{n41}
\xymatrix{0 \ar[r]& H^{*-2}_{T}(\tY_{k-1})\ar[r]^-{~(\jmath_k)_*~} & H^*_T(\tY_{k})\ar[r]\ar[rd]_-{\rho_k}& H^*_T(\tY_k-\tY_{k-1}) \ar[r]\ar[d]^{\cong}  &0 \\ &&& H^*_T(\pr_2^{-1}(\sigma_k))&}\eeq
	for each $1\leq k\leq r$, which is split. Indeed, $\jmath_k^*\circ(\jmath_k)_*$ is equal to the multiplication by $\a_k:=e^T(N_{H_{k-1}/H_k})=(\pm(t_{j+k}-t_{j+i}))_{(v,\sigma_i)\in \tY_{k-1}^T}$ which is not a zero divisor. 
	
	The short exact sequence \eqref{n41} gives us the isomorphism 
	\[H^*_T(\tY_{\bh})\cong \bigoplus_{i=0}^r H^{*-2k}_T(\pr_2^{-1}(\sigma_k))\]
	as graded vector spaces. 
	Furthermore, there is an explicit isomorphism
	\beq \label{51} Y_h\cong \tY_0 \xrightarrow{~\cong ~}\pr_2^{-1}(\sigma_k)=\{y\in Y_{h_+}:f_\bh(y)\in \C e_k\}\eeq 
	which sends $Tg$ to $Tg\cdot (j,j+k)$ where $(j,j+k)$ denotes the permutation matrix in $U(n)$ associated to the transposition $(j,j+k)\in S_n$.
	Hence, it remains to show that $(\jmath_k)_*$ and $\rho_k$ are $S_n$-equivariant.  
	To see this, observe that \eqref{n41} fits into the commutative diagram
	\[\xymatrix{0 \ar[r]& H^{*-2}_{T}(\tY_{k-1})\ar[r]^{(\jmath_k)_*} \ar[d]_-{\a_k\circ \res }  & H^*_T(\tY_{k})\ar[r]^-{\rho_k}\ar[d]_-{\res}& H^*_T(\pr_2^{-1}(\sigma_k)) \ar[r]\ar[d]_-{\res} &0 \\ 0\ar[r]&H^{*}_{T}(\tY_{k-1}^T)\ar[r]& H^*_T(\tY_{k}^T)\ar[r]^-{\rho_k^T}& H^*_T(\pr_2^{-1}(\sigma_k)^T)\ar[r]&0}\]
	of short exact sequences, where the vertical maps are all injective and the bottom row is induced by $\tY_{k}^T=\tY_{k-1}^T \sqcup \pr_2^{-1}(\sigma_k)^T $. 
	Therefore, $(\jmath_k)_*$ is $S_n$-equivariant since $\a_k$ is $S_n$-equivariant and $\tY_{k-1}^T$ is $S_n$-invariant in $\tY_k^T$. 
		
	Moreover, $\rho_k$ is $S_n$-equivariant since $\rho_k^T$ is the projection map induced by the inclusion $S_n\times \{\sigma_k\} \hookrightarrow  S_n\times \{\sigma_i\}_{i\leq k}$, which is $S_n$-equivariant.	
	On the other hand, the isomorphism \eqref{51} is $T$-equivariant with respect to the usual $T$-action on $Y_h$ composed with the interchange of $t_j$ and $t_{j+k}$. 
	This completes the proof for $\bh$ of type (1).
	
	For $\bh$ of type (2), consider the coordinate planes $H_k\cong \PP^{k}$ in $\PP^r$ spanned by the coordinate points $\sigma_r,\cdots, \sigma_{r-k+1}$ and the induced $T$-equivariant filtration $Y_h\cong \tY_0 \subset \cdots \subset \tY_r=\tY_\bh$ 
	of $\tY_\bh$ given by $\tY_k=\pr_2^{-1}(H_k)$. Then the proof is parallel to that for $\bh$ of type (1) and we omit the detail.	
\end{proof}

\bigskip

\section{Unicellular LLT polynomials and twin manifolds}\label{S4}
		
Let $\bh=(h_-,h,h_+)$ be a triple of Hessenberg functions in Definition \ref{43}.
Let $\tY_\bh$ denote the roof manifold which is a $Y_h$-fiber bundle over $\PP^r$ by Proposition \ref{38} and also the blowup of $Y_{h_+}$ along $Y_{h_-}$ by Proposition \ref{39}. 
 
In this section, we will apply the geometry of the twin manifolds associated to $\bh$ to compare the cohomology of the twin manifolds. 
In particular, we will establish the modular law \eqref{5}
for a modular triple $\bh$ (when $r=1$). 

\medskip

\subsection{The modular law}

By Proposition \ref{n35}, we can consider the Frobenius characteristic of the cohomology of the roof manifold $\tY_\bh$. 
\begin{definition}\label{28}
	Following Definition~\ref{30}, we let
	\[\cP(\bh):=\sum_{k\geq 0}\ch(H^{2k}(\tY_\bh))q^k \in \Lambda[q].\]
\end{definition}

For every triple $\bh=(h_-,h,h_+)$ in Definition~\ref{43}, the following is immediate from \eqref{n33}, \eqref{17} and Proposition \ref{n35}.
\begin{proposition}\label{27}
	\beq \label{31} \cP(\bh)=\cP(h_+)+q[r]_q\cP(h_-).\eeq
	\beq \label{32}\cP(\bh)=[r+1]_q\cP(h).\eeq
\end{proposition}

Combining \eqref{31} and \eqref{32}, we have the following. 

\begin{theorem} \label{2}
	Let $\bh=(h_-,h,h_+)$ be a triple in Definition~\ref{43}. Then, 
	\[[r+1]_q \cP(h)=\cP(h_+)+q[r]_q \cP(h_-).\]
	In particular, the modular law
	\beq \label{53} [2]_q\cP(h)=\cP(h_+)+q\cP(h_-) \eeq
	holds for a modular triple $\bh$ (when $r=1$).
\end{theorem}

Moreover, we have canonical $S_n$-equivariant isomorphisms
\beq \label{n43} H^*(Y_h)\oplus H^{*-2}(Y_h)\cong H^*(\tY_\bh)\cong 
H^*(Y_{h_+})\oplus H^{*-2}(Y_{h_-})\eeq
for a modular triple $\bh$.

\subsection{Unicellular LLT and twins} 
The chromatic quasisymmetric functions and the representations of $S_n$ on the cohomology of Hessenberg varieties are related by the Shareshian-Wachs conjecture \cite{SW}, proved in \cite{BC,GP2},  
\beq \label{57}\cF(h):=\sum_{k\geq 0}\ch(H^{2k}(X_h))q^k=\w\,\csf_h(q)\eeq
where the $S_n$-action on $H^*(X_h)$ is given by the dot action \eqref{19} and
$\w$ denotes the involution of $\Lambda$ which interchanges each Schur function with its transpose.

An analogous connection between unicellular LLT polynomials and the representations of $S_n$ on the cohomology of twin manifolds was discovered by Masuda-Sato and Precup-Sommers. 
\begin{theorem} \label{48} \cite[Proposition 3.2.1]{MS} \cite[Corollary 7.9 (2)]{PS}
	$$\cP(h)=\LLT_h(q)$$ for every Hessenberg function $h$.
\end{theorem}
\begin{proof}
	By the characterization of $\LLT_h$ in Theorem~\ref{46}, it suffices to show that $\cP$ satisfies the three conditions (1), (2) and (3). (1) was proved in Proposition \ref{n16} and (2) was proved in Proposition \ref{58}. Finally, Theorem \ref{2} proves (3). 
\end{proof}

\begin{remark}\label{56}
	Masuda-Sato's proof of Theorem~\ref{48} in \cite{MS} makes an essential use of the Shareshian-Wachs conjecture while our proof is based on the geometry of twin manifolds, independent of the SW conjecture. 
	Their proof is a direct consequence of the following three major ingredients:
	\begin{enumerate}
		\item[(i)] the Shareshian-Wachs conjecture  \eqref{57},
		\item[(ii)] the Carlsson-Mellit relation (cf.~\cite[Proposition 3.5]{CM})
		\[{(q-1)^n}\csf_h(q)=\LLT_h[{(q-1)}X;q]\]
		where $X=x_1+x_2+\cdots$ with variables $x_i$ of $\Lambda$ and $[~]$ denotes the plethystic substitution, and 
		\item[(iii)] the parallel plethystic relation (cf.~\cite[Proposition 3.0.2]{MS})
		\[{(1-q)^n }\cF(h)=\cP(h)\left[(1-q)X;q\right]\]
		obtained by applying the formula in \cite[Proposition 3.3.1]{Hai} to the isomorphism   
\eqref{n49}. 		
	\end{enumerate}
	Pictorially, it can be summarized in the diagram
	\beq
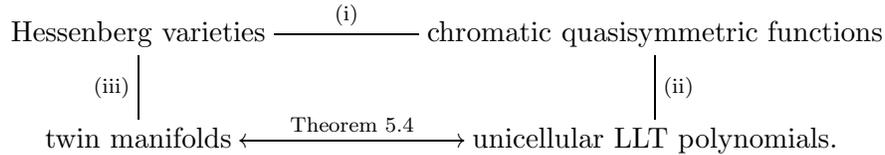
\label{n60}\xymatrix{\text{Hessenberg varieties}\ar@{-}[rr]^-{\textrm{(i)}}\ar@{-}[d]_-{\textrm{(iii)}} && \text{chromatic quasisymmetric functions}\ar@{-}[d]^-{\textrm{(ii)}}\\ 
	\text{twin manifolds} \ar@{<->}[rr]^-{\text{Theorem~\ref{48}}} && \text{unicellular LLT polynomials}.}\eeq
	Note that  our proof of Theorem \ref{48} is direct without relying on (i), (ii) or (iii).	
	
	On the other hand, in \cite{PS}, Precup-Sommers studied a relation between the coefficients of $\cP(h)$ in the Schur basis expansion, which are polynomials in $q$, and the cohomology of \emph{nilpotent} Hessenberg varieties (cf.~\cite[\S1]{PS}). Based on this, they proved the modular law \eqref{53}, by establishing the modular law for the nilpotent Hessenberg varieties. 
	However, their work does not address the geometry of twin manifolds.
\end{remark}

\begin{remark}
	Our proof of Theorem \ref{48} together with (ii) and (iii) in \eqref{n60} provide us with a new proof of the Shareshian-Wachs conjecture \eqref{57}.
\end{remark}

\end{document}